\documentclass[11pt]{amsart}

\usepackage{amssymb, amsmath, amsthm,amssymb,amsfonts,latexsym,mathtools}
\usepackage{tikz-cd}
\usepackage{extarrows}
\usepackage{mathrsfs}
\usepackage{arydshln}
\usepackage{bbm}
\usepackage[square,numbers,sort&compress]{natbib}

\usepackage{multirow}
\usepackage{hyperref}

\usepackage{enumerate}
\usepackage{booktabs}
\usepackage{pdfsync}

\newcommand{\xra}{\xrightarrow}
\newcommand{\Z}{\mathbb{Z}}
\newcommand{\ZF}{\mathbb{Z}_\mathbb{F}}
\newcommand{\zp}[1]{\Z/p^{#1}\Z}

\newcommand{\NN}{\mathbb{N}\cup\{\infty\}}

\newcommand{\R}{\mathbb{R}}

\newcommand{\FP}[1]{\mathbb{F}P^{#1}}

\newcommand{\id}{\mathbbm{1}}

\newcommand{\A}[1]{\ensuremath{\mathcal{A}_\sharp^{#1}}}
\newcommand{\Ah}[1]{\ensuremath{\mathcal{A}_\ast^{#1}}}

\newcommand{\E}{\ensuremath{\mathcal{E}}}
\newcommand{\NE}{\ensuremath{N\E}}
\newcommand{\NEh}{\ensuremath{N_\ast\E}}

\newcommand{\Pave}{Pave\v{s}i\'{c}\,}

\DeclareMathOperator{\Hom}{Hom}
\DeclareMathOperator{\End}{End}
\DeclareMathOperator{\Aut}{Aut}

\theoremstyle{plain}
\newtheorem{theorem}{Theorem}[section]
\newtheorem{proposition}[theorem]{Proposition}
\newtheorem{corollary}[theorem]{Corollary}
\newtheorem{lemma}[theorem]{Lemma}

\theoremstyle{definition}
\newtheorem{definition}[theorem]{Definition}
\newtheorem{example}[theorem]{Example}

\theoremstyle{remark}

\title{Self-closeness numbers of product spaces}

\author{Pengcheng Li}
\email{lipc@sustech.edu.cn}

\address{Department of Mathematics, Southern University of Science and Technology, Shenzhen, Guangdong, 518055, China}


\begin{document}

\begin{abstract}
	The self-closeness number of a CW-complex is a homotopy invariant defined by the minimal number $n$ such that every self-maps of $X$ which induces automorphisms on the first $n$ homotopy groups of $X$ is a homotopy equivalence. In this article we study the self-closeness numbers of finite Cartesian products, and prove that under certain conditions (called reducibility), the self-closeness number of product spaces equals to the maximum of self-closeness numbers of the factors. A series of criteria for the reducibility are investigated, and the results are used to determine self-closeness numbers of product spaces of some special spaces, such as Moore spaces, Eilenberg-MacLane spaces or atomic spaces.
\end{abstract}

\maketitle

\section{Introduction}\label{sec:intro}

Given based spaces $X,Y$, denote by $[X,Y]$ the set of homotopy classes of based maps from $X$ to $Y$, and let $\E(X)$ denote the group of homotopy classes of self-homotopy equivalences of $X$.  In 2015, Choi and Lee introduced a numerical homotopy invariant,
called the \emph{self-closeness number} $\NE(X)$ of $X$, by the minimal non-negative integer $n$ such that $\A{n}(X)=\E(X)$, where 
\[\A{n}(X)\coloneqq \{f\in [X,X]:f_\sharp\colon\pi_k(X)\xra{\cong}\pi_k(X),~~\forall ~k\leq n\}.\]
The self-closeness number provides a useful method to study self-homotopy equivalences by focusing on the homotopy groups of the space in certain range. Since Choi and Lee, this homotopy invariant has been studied by several authors, such as Oda and Yamaguchi  \cite{OY17,OY18,OY20}, Li \cite{L21}. 
This paper is devoted to the study of self-closeness numbers of finite product spaces. 

The group of self-homotopy equivalences of product spaces can be naturally studied from maps between its factors, related research papers which adopted this method include \cite{AY82,BH90,Heath96,Pavesic99}. In particular, \Pave \cite{Pavesic99} showed that under certain ``diagonalizability'' or ``reducibility'' conditions, the group $\E(X\times Y)$ can be decomposed as a product of its subgroups $\E(X\times Y;I)$, which consisting of elements of $\E(X\times Y)$ leaving $I$ fixed, $I=X,Y$. Many computable criteria for the reducibility appeared in the later paper \cite{Pavesic07}. 
By the universal property of Cartesian products, a map $f\colon Z\to X\times Y$ of spaces is usually denoted component-wise as $f=(f_X,f_Y)$,
For a fixed $n\in\NN$, we say the monoid $\A{n}(X\times Y)$ is \emph{reducible} if $f\in\A{n}(X\times Y)$ implies so are $f_{XX}$ and $f_{YY}$. In this paper we utilize the \Pave's analysis in \cite{Pavesic99,Pavesic07} to study the factorizations and the reducibility of the monoids $\A{n}(X\times Y)$. Although the obtained factorizations of the monoids $\A{n}(X\times Y)$ can be regarded as generalizations of that of $\E(X\times Y)$, there seems more restrictions on the criteria for the reducibility of monoids $\A{n}(X\times Y)$ than that for the group $\E(X\times Y)$. This is reasonable because elements of $\A{n}(X\times Y)$ usually have no homotopy inverses, and some homomorphisms between homotopy groups are not induced by maps between spaces.

The paper is organized as follows. In Section \ref{sec:def-redu} we extend the concept of the reducibility of $\A{n}(X\times Y)$ to that of $\A{n}(X_1\times\cdots\times X_m)$ for any $m$, and discuss some simple situations where the reducibility can be easily verified. In Section \ref{sec:decomp} we utilize the techniques in \cite{Pavesic07} to prove that if $\A{N}(X_1\times\cdots\times X_m)$ is reducible with $N=\max\{\NE(X_1),\cdots,\NE(X_m)\}$, then $N=\NE(X_1\times\cdots\times X_m)$ (Theorem \ref{thm:scn-prod}). Section \ref{sec:redu} further investigate sufficient conditions for the reducibility of the monoids $\A{n}(X\times Y)$. Firstly we prove that if every self-map of $Y$ that factors through $X$ induces nilpotent and central endomorphism of the first $n$ homotopy groups of $Y$, $n\geq \NE(X)$, then $\A{n}(X\times Y)$ is reducible if and only if $f\in\A{n}(X\times Y)$ implies that $f_{XX}\in\E(X)$ (Theorem \ref{thm:redu-prod}). Here an induced endomophism $\pi_k(f)$ is \emph{central} means that it commutes with endomorphisms induced by any other self-maps. A series of useful variations of Theorem \ref{thm:redu-prod} are derived.
Section \ref{sec:Appl} serves as applications of the criteria developed in Section \ref{sec:redu}. We firstly obtain criteria for the reducibility of $\A{n}(X\times Y)$ with $X$ a Moore space or an Eilenberg-MacLane space. Some examples are computed. Using the atomicity of spaces, we then determine the self-closeness number of finite products of atomic spaces under certain assumptions, see Theorem \ref{thm:atomic}. 

Throughout the paper all spaces have the homotopy types of connected CW-complexes, and all maps are base-point-preserving and are identified with their homotopy classes.
Given a group $G$, denote by $\End(G)$ the monoid of endomorphisms of $G$, and let $\Aut(G)$ be the units of $\End(G)$.

\section*{Acknowledgement}
The author was partially supported by National Natural Science Foundation of
China (Grant No. 12101290) and Project funded by China Postdoctoral
Science Foundation (Grant No. 2021M691441).

\section{Reducibility of the monoids $\A{n}(X_1\times\cdots\times X_m)$}\label{sec:def-redu}
Let $X$ be a connected CW-complex.  For each $n\geq 0$, the submonoid $\A{n}(X)$ of $[X,X]$ is given by 
\[\A{n}(X)\coloneqq \{f\in [X,X]:\pi_k(f)\in\Aut(\pi_k(X)),~forall~k\leq n\}.\]

Let $m\geq 2$ be a fixed integer. For each $k=1,\cdots,m$, denote by $p_k=p_{X_k}\colon X_1\times\cdots\times X_m\to X_k$ and $i_k=i_{X_k}\colon X_k\to X_1\times\cdots\times X_m$ the canonical projections and inclusions (defined by base-points). Then $p_k\circ i_l=\delta_{k,l}\cdot \id_{X_k}$, where $\delta_{k,l}$ is the Kronecker delta. Write a self-map $f$ of $X_1\times\cdots\times X_m$ component-wise as $f=(f_1,\cdots,f_m)$, where $f_k=p_{X_k}\circ f$. We have maps $f_{kl}=f_k\circ i_l\colon X_l\to X_k$.
The concept of \emph{reducibility} of self-homotopy equivalences of finite products $X_1\times\cdots\times X_m$ (cf.\cite{Pavesic07}) can be extended to that of elements of the monoids $\A{n}(X_1\times\cdots\times X_m)$ as follows.

\begin{definition}\label{def:prod-redu}
For any fixed $1\leq n\leq \infty$, $\A{n}(X_1\times\cdots\times X_m)$ is said to be \emph{reducible} if $f=(f_1,\cdots,f_m)\in \A{n}(X_1\times\cdots\times X_m)$ implies that so are the self-maps $(f_1,\cdots,f_m)$ with one component (and hence any number of components) $f_k$ replaced by $p_k$, $k=1,\cdots,m$.
\end{definition}

Note that the reducibility of $\A{n}(X\times Y)$ corresponds to \emph{$n$-reducibility} of self-maps of $X\times Y$ in \cite{JL17}.
 The following lemma is clear by definition. 

\begin{lemma}\label{lem:redu-equiv}
	For any fixed $1\leq n\leq \infty$, $\A{n}(X_1\times\cdots\times X_m)$ is reducible if and only if $f\in\A{n}(X_1\times\cdots\times X_m)$ implies that $f_{ii}\in\A{n}(X_i)$, $i=1,\cdots,m$.
\end{lemma}

\begin{lemma}\label{lem:redu-trivial}
	If for all $1\leq i<j\leq m$ and all $k\leq n$, every map $X_i\to X_j$ induces a trivial homomorphism $\pi_k(X_i)\to \pi_k(X_j)$ or every $X_j\to X_i$ induces a trivial homomorphism $\pi_k(X_j)\to \pi_k(X_i)$, then $\A{n}(X_1\times\cdots\times X_m)$ is reducible.
	\begin{proof}
		Given $f\in\A{n}(X_1\times\cdots\times X_m)$, for any $k\leq n$, $\pi_k(f)$ can be presented by the following matrix:
		\[[\pi_k(f_{ij})]_{m\times m}=\begin{bmatrix}
			\pi_k(f_{11}) & \cdots &\pi_k(f_{1m})\\
			\vdots & \ddots & \vdots \\
			\pi_k(f_{m1}) & \cdots & \pi_k(f_{mm})
		\end{bmatrix}\]
		By assumption, the matrix $\pi_k(f)$ is upper-triangular or lower-triangular, hence $\pi_k(f)$ is invertible if and only if $\pi_k(f_{ii})\in\Aut(\pi_k(X_i))$, $i=1,\cdots,m$. Thus $f_{ii}\in\A{n}(X_i)$, $i=1,\cdots,m$. It then follows by Lemma \ref{lem:redu-equiv} that $\A{n}(X_1\times\cdots\times X_m)$ is reducible.
	\end{proof}
\end{lemma}
Some situations where Lemma \ref{lem:redu-trivial} applies were summarized in \citep[Corollary 2.2]{Pavesic99}, we don't repeat them here.
For a group $G$, denote by $Z(G)$ the center of $G$. Bidwell, Curran and McCaughan \citep[Theorem 3.2]{BCM06} proved that if $G$ and $H$ are finite groups with no common direct factors,  the group $\Aut(G\times H)$ consists of elements of the form 
\[\begin{bmatrix}
	\alpha&\beta\\
	\gamma&\delta
\end{bmatrix},\]
where $\alpha\in\Aut(G),\delta\in\Aut(H),\beta\in \Hom(H,Z(G)),\gamma\in \Hom(G,Z(H))$. Inductively applying their theorem we get

\begin{proposition}\label{prop:BCM}
	If for each $k\leq n$ and each $i\leq m$, $\pi_k(X_i)$ is finite, and for each pair $1\leq i\neq j\leq m$, $\pi_k(X_i)$ and $\pi_k(X_j)$ has no common direct factors, then $\A{n}(X_1\times\cdots\times X_m)$ is reducible.
\end{proposition}

\section{Factorization of the monoids $\A{n}(X_1\times\cdots\times X_m)$}\label{sec:decomp}
Given a monoid $(M,\cdot, 1)$ and its two submonoids $M_1,M_2$, the product $M_1\cdot M_2\coloneqq\{m_1\cdot m_2:m_i\in M_i,~i=1,2\}$ is a submonoid of $M$ if and only if $M_2\cdot M_1\subseteq M_1\cdot M_2$. Hence, if $M\subseteq M_1\cdot M_2$ holds as sets, then $M=M_1\cdot M_2$ holds as monoids. Unless otherwise stated, all factorizations of monoids are of submonoids. 

For each $n$, denote by $\A{n}(X\times Y;X)$ the submonoid of $\A{n}(X\times Y)$ consisting of elements $f$ satisfying $p_X\circ f\simeq p_X$, where $p_X\colon X\times Y\to X$ is the canonical projection. $\A{n}(X\times Y;Y)$ is similarly defined. If $n=\infty$, substitute $\A{\infty}$ by $\E$. Note that  $\E(X\times Y;I)$ corresponds to the notations $\E_I(X\times Y)$ in \cite{Pavesic99}, $I=X,Y$. 

\begin{proposition}\label{prop:prod-multi}
	Let $X,Y,Z$ be CW-complexes.
	\begin{enumerate}[(1) ]
		\item\label{F1} If $\A{n}(X\times Y)$ is reducible for some $n\geq \NE(X)$, then there is a factorization 
		 \[\A{n}(X\times Y)=\A{n}(X\times Y;X)\cdot \E(X\times Y;Y).\]
		\item\label{F2} If $\A{n}(X\times Y\times Z)$ is reducible for some $n\geq \NE(Y)$, then there is a factorization 
		\[\A{n}(X\times Y\times Z;X)=\A{n}(X\times Y\times Z;X\times Y)\cdot \E(X\times Y \times Z;X\times Z).\]
	\end{enumerate}
	\begin{proof}
		It suffices to show that both the left inclusions $``\subseteq''$ of the equalities of monoids in (1) and (2) hold as sets.

(1) Suppose $f=(f_X,f_Y)\in \A{n}(X\times Y)$. Then by assumption we have 
\[(p_X,f_Y)\in\A{n}(X\times Y;X),\quad (f_X,p_Y)\in\A{n}(X\times Y;Y)=\E(X\times Y;Y).\]
Let $(g,p_Y)$ be the inverse of $(f_X,p_Y)$. Then \[(p_X,f_Y)\circ (g,p_Y)=(g,f_Y(g,p_Y))\in\A{n}(X\times Y)\] and the reducibility of $\A{n}(X\times Y)$ implies $(p_X,f_Y(g,p_Y))\in\A{n}(X\times Y;X)$.
Thus
\[(p_X,f_Y(g,p_Y))\circ (f_X,p_Y)=\big(f_X,f_Y(g,p_Y)(f_X,p_Y)\big)=(f_X,f_Y)\]
and therefore the factorization in (1) is proved.

(2) Suppose $(p_X,f_Y,f_Z)\in\A{n}(X\times Y\times Z;X)$. Since $n\geq \NE(Y)$,
the inverse $(p_X,f_Y,p_Z)^{-1}$ exists. Apply the reducibility of $\A{n}(X\times Y\times Z)$ to the composition  
	\[(p_X,p_Y,f_Z)\circ (p_X,f_Y,p_Z)^{-1}=(p_X,p_Y(p_X,f_Y,p_Z)^{-1},f_Z(p_X,f_Y,p_Z)^{-1}),\]
    we have $(p_X,p_Y,f_Z(p_X,f_Y,p_Z)^{-1})\in\A{n}(X\times Y\times Z;X\times Y).$
	Thus \[(p_X,f_Y,f_Z)=(p_X,p_Y,f_Z(p_X,f_Y,p_Z)^{-1})\circ (p_X,f_Y,p_Z),\] 
    which completes the proof.
	\end{proof}
\end{proposition}

Recall that there is a chain of submonoids by inclusion: 
\[\E(X)=\A{\infty}(X)\subseteq \A{n}(X)\subseteq \A{1}(X)\subseteq \A{0}(X)=[X,X].\]
The \emph{self-closeness number} $\NE(X)$ is defined by 
\[\NE(X)\coloneqq\min\{\A{n}(X)=\E(X)\}.\]
Choi and Lee \cite{CL15} proved that for the product space $X\times Y$, there holds an inequality (\citep[Theorem 3]{CL15}):
\[\NE(X\times Y)\geq \max\{\NE(X),\NE(Y)\}.\]
Inductively applying this inequality we get
\begin{lemma}\label{lem:ineq-HK}
	Let $X_1,\cdots,X_m$ be CW-complexes. There holds an inequality 
	\[\NE(X_1\times\cdots\times X_m)\geq \max\{\NE(X_1),\cdots,\NE(X_m)\}.\]
\end{lemma}

\begin{theorem}\label{thm:scn-prod}
	Let $X_1,\cdots,X_m$ be based connected CW-complexes and let $N=\max\{\NE(X_1),\cdots,\NE(X_m)\}$.
	If $\A{N}(X_1\times\cdots\times X_m)$ is reducible, then $\A{N}(X_1\times\cdots\times X_m)=\E(X_1\times\cdots\times X_m)$, and hence $N=\NE(X_1\times\cdots\times X_m)$.
 \begin{proof}
	By Lemma \ref{lem:ineq-HK}, it suffices to show that $\NE(X_1\times\cdots\times X_m)\leq N$,  or equivalently $\A{N}(X_1\times\cdots\times X_m)\subseteq \E(X_1\times\cdots\times X_m)$.

	For each $2\leq k\leq m$, denote by 
	\begin{equation*}
		\varPi_k=X_1\times\cdots\widehat{X}_k\times \cdots\times X_m
	\end{equation*}
	 the subspace of $\prod_{k=1}^mX_k$ whose $k$-th coordinate is the base-point of $X_k$. By Proposition \ref{prop:prod-multi} (\ref{F2}), for each $2\leq k\leq m$ there exist a factorization 
	\begin{multline*}
		\A{N}(X_1\times\cdots\times X_m;X_1\times\cdots\times X_{k-1})\\
		=\A{N}(X_1\times\cdots\times X_m;X_1\times\cdots\times X_k)\cdot \E(X_1\times\cdots\times X_m;\varPi_k).
	\end{multline*} 
	There holds a sequence of equalities
	\begin{align*}
		\A{N}(X_1\times\cdots\times X_m)&=_1\A{N}(X_1\times\cdots\times X_m;X_1)\cdot \E(X_1\times\cdots\times X_m;\varPi_1)\\
		&=_2\A{N}(X_1\times\cdots\times X_m;X_1\times X_2)\cdot \E(X_1\times\cdots\times X_m;\varPi_2)\\
		&\hspace{5cm}
		\cdot \E(X_1\times\cdots\times X_m;\varPi_1)\\
		&\vdots\\
		&=_m\prod_{i=1}^m\E(X_1\times\cdots\times X_m;\varPi_{m+1-i})\\
		&=\E(X_1\times\cdots\times X_m),
	\end{align*}
	where the first equality $=_1$ holds by Proposition \ref{prop:prod-multi} (\ref{F1}), the $i$-th equalities $=_i (i=2,\cdots,m)$ hold by the Proposition \ref{prop:prod-multi} (\ref{F2}), and the last equality is then clear. Thus $\NE(X_1\times\cdots\times X_m)\leq N$ and therefore $\NE(X_1\times\cdots\times X_m)=N$.
 \end{proof}	
\end{theorem}

Given a group $G$ and an integer $n\geq 1$, denote by $K_n(G)$ the Eilenberg-MacLane space satisfying $\pi_n(K_n(G))\cong G$, $\pi_{i\neq n}(K_n(G))=0$.
It is clear that $\NE(K_n(G))=n$ for any $n,G$.
\begin{example}\label{ex:sphere-EM}
	Let $n_1,\cdots,n_m$ be positive integers and let $G_1,\cdots,G_m$ be groups. 
	\begin{enumerate}[(1) ]
		\item If $n_1<\cdots<n_m$, then $\NE(S^{n_1}\times\cdots\times S^{n_m})=n_m$.
		\item\label{ex:EM} $\NE\big(K_{n_1}(G_1)\times\cdots\times K_{n_m}(G_m)\big)=\max\{n_1,\cdots,n_m\}$.
	\end{enumerate} 
	
	\begin{proof}
		 If $n_1<\cdots<n_m$, by Lemma \ref{lem:redu-trivial} it is clear that $\A{r}(S^{n_1}\times\cdots \times S^{n_m})$ and $\A{r}\big(K_{n_1}(G_1)\times\cdots\times K_{n_m}(G_m)\big)$ is reducible for any $r\in\NN$. Thus by Theorem \ref{thm:scn-prod} we get the equalities.
		
	For Eilenberg-MacLane spaces, note that if $k=l$, then $K_{k}(G)\times K_{l}(H)\simeq K_{k}(G\times H)$ for any groups $G,H$. It follows that the second equality still holds even if there exist some $i\neq j$ such that $n_i=n_j$.
 	\end{proof}
\end{example}

There is an alternative proof of Theorem \ref{thm:scn-prod} motivated by the \emph{LU}-decomposition of $\E(X_1\times\cdots\times X_m)$ (\citep[Theorem 5.4]{Pavesic07}). For each $k=1,\cdots,m$, let $l_k,u_k$ be the self-maps of $X_1\times\cdots\times X_m$ defined by 
\[l_k=(p_1,\cdots,p_k,\ast,\cdots,\ast),u_k=(\ast,\cdots,\ast,p_k,\cdots,p_m),\]
where $p_i=p_{X_i}$ is the canonical projection onto the $i$-th factor, $i=1,\cdots,m$, and $\ast$ denote the constant maps. Denote a self-map $f$ of $X_1\times\cdots\times X_m$ by $f=(f_1,\cdots,f_m)$ with $f_k=p_k\circ f$, and denote $f_{kl}=p_{X_k}\circ f\circ i_{X_l}$, $k,l=1,\cdots,m$.
For each $n\in\NN$, set
\begin{align*}
	L_\sharp^n(X_1,\cdots,X_m)&\coloneqq\{f\in\A{n}(X_1\times\cdots\times X_m):f_k=f_k\circ l_k,k=1,\cdots,m\},\\
	U(X_1,\cdots,X_m)&\coloneqq\{f\in\E(X_1\times\cdots\times X_m):f_k=f_k\circ u_k,k=1,\cdots,m\}.
\end{align*}
Note that endomorphisms induced by elements of $L_\sharp^n(X_1,\cdots,X_m)$ are represented by lower-triangular matrices, while that of $U(X_1,\cdots,X_m)$ have the form of upper-triangular matrices with identities on the diagonal entries.
Note also that $L(X_1,\cdots,X_m)\coloneqq L_\sharp^{\infty}(X_1,\cdots,X_m)$ and $U(X_1,\cdots,X_m)$ are subgroups of $\E(X_1\times\cdots X_m)$ (\citep[Proposition 5.3]{Pavesic07}).

\begin{lemma}\label{lem:LU}
	Let $n\in\NN$. 
	\begin{enumerate}[(1)]
		\item\label{LU-1} $L_\sharp^n(X_1,\cdots,X_m)$ is a submonoid of $\A{n}(X_1\times\cdots\times X_m)$.
		\item\label{LU-2} There is a split extension of monoids: 
	     \[\begin{tikzcd}[column sep=small]
			1\ar[r]& \bar{L}_\sharp^n(X_1,\cdots,X_m)\ar[r]&L_\sharp^n(X_1,\cdots,X_m)\ar[r,"\Phi", shift left]&\prod_{k=1}^m \A{n}(X_k)\ar[r]\ar[l,"s",shift left]&1,
		 \end{tikzcd}\]
	     where $\Phi(f)=(f_{11},\cdots,f_{mm})$, $s(g_1,\cdots,g_m)$ is defined by $p_k(s(g_1,\cdots,g_m))=g_k\circ p_k$, $k=1,\cdots,m$; and 
		 \[\bar{L}_\sharp^n(X_1,\cdots,X_m)\coloneqq\ker(\Phi)=\{f\in L_\sharp^n(X_1,\cdots,X_m):f_{kk}=\id_{X_k},k=1,\cdots,m\}.\]
	\end{enumerate} 
	\begin{proof}
	The proof of (\ref{LU-1}) is similar to that of \citep[Proposition 3.1]{Pavesic05}, and the proof of (\ref{LU-2}) refers to that of \citep[Proposition 5.5]{Pavesic07}.
	\end{proof}
\end{lemma}

There hold two chains by inclusions of monoids:
\begin{align*}
	L_\sharp^n(X_1,\cdots,X_m)&\subseteq L_\sharp^n(X_1\times X_2,\cdots,X_m)\subseteq \cdots\subseteq L_\sharp^n(X_1\times\cdots\times X_m);\\
	U(X_1,\cdots,X_m)&\supseteq U(X_1\times X_2,\cdots,X_m)\supseteq \cdots \supseteq U(X_1\times\cdots\times X_m)=\{1\}.
\end{align*}

We have the following extension of \citep[Theorem 5.4]{Pavesic07}.
\begin{proposition}\label{prop:LU}
If $n\geq \max\{\NE(X_1),\cdots,\NE(X_{m-1})\}$ and $\A{n}(X_1\times\cdots\times X_m)$ is reducible, then there is a factorization of monoids 
	\[\A{n}(X_1\times\cdots\times X_m)=L_\sharp^n(X_1,\cdots,X_m)\cdot U(X_1,\cdots,X_m).\]
	
	\begin{proof}
	We only sketch the proof here, the details are similar to that of \citep[Theorem 5.4]{Pavesic07}. 
	
	The induction process starts with $m=2$. Suppose $f=(f_1,f_2)\in\A{n}(X_1\times X_2)$. By reducibility we have $(f_1,p_2)\in \A{n}(X_1\times X_2)$, which is equivalent to $f_{11}\in \A{n}(X_1)=\E(X_1)$ for $n\geq \NE(X_1)$. Hence $(f_1,p_2),(f_{11}\circ p_1,p_2)\in L_\sharp^n(X_1,X_2)$. 
	By the arguments in the top $5$ lines of \citep[Page 410]{Pavesic07} we get that \begin{align*}
		 f'&=f\circ (f_1,p_2)^{-1}\circ (f_{11}\circ p_1,p_2)\in L_\sharp^n(X_1,X_2),\\
		 f''&=(f_{11}\circ p_1,p_2)^{-1}\circ (f_1,p_2)\in U(X_1,X_2)
	\end{align*}
	Thus $f=f'\circ f''$, the factorization in the proposition is proved for $m=2$.

	The arguments for the general case is totally parallel to the last paragraph of the proof of \citep[Theorem 5.4]{Pavesic07}, by substituting the notations ``$\Aut$'' by $\A{n}$, and ``$L$'' by ``$L_\sharp^n$''.  
	\end{proof}
\end{proposition}

\begin{proof}[Alternative proof of Theorem \ref{thm:scn-prod}]
	Let $N=\max\{\NE(X_1),\cdots,\NE(X_m)\}$.
	By Lemma \ref{lem:LU} (\ref{LU-2}), for any $n\geq N$ there hold 
	\[L_\sharp^n(X_1,\cdots,X_m)=L(X_1,\cdots,X_m)\subseteq \E(X_1\times\cdots\times X_m).\] 
	Thus by Proposition \ref{prop:LU} we get \[\A{N}(X_1\times\cdots\times X_m)=L(X_1,\cdots,X_m)\cdot U(X_1,\cdots,X_m)=\E(X_1\times\cdots\times X_m).\]
\end{proof}

For a simply-connected CW-complex $X$ and each $n\in\NN$, denote 
\[\Ah{n}(X)\coloneqq\{f\in [X,X]:H_i(f)\in \Aut(H_i(X)),~\forall ~i\leq n\}.\]
There is also the \emph{homology self-closeness number} $\NEh(X)$ defined by 
\[\NEh(X)\coloneqq\min\{n:\Ah{n}(X)=\E(X)\}.\]
For related papers, one may consult \cite{OY20,L21}.
Given spaces $X_1,\cdots,X_m$, denote by $i_k\colon X_k\to X_1\vee \cdots\vee X_m$ the canonical inclusion. Write a self-map $f\colon X_1\vee \cdots\vee X_m\to X_1\vee \cdots\vee X_m$ component-wise as $f=(f_1,\cdots,f_m)$, where $f_k=f\circ i_k$, $k=1,\cdots,m$. For a fixed $n\in\NN$, we say that $\Ah{n}(X_1\vee \cdots\vee X_m)$ is \emph{reducible} if $f=(f_1,\cdots,f_m)\in\Ah{n}(X_1\vee \cdots\vee X_m)$ implies that so are the self-maps $(f_1,\cdots,f_m)$ with one component (and hence any number of components) $f_k$ replaced by $i_k$, $k=1,\cdots,m$. 

We remark without proof that Theorem \ref{thm:scn-prod} has the following dualization.
\begin{theorem}\label{thm:scn-sum}
	Let $X_1,\cdots,X_m$ be simply-connected based CW-complexes. If $\Ah{N}(X_1\vee \cdots\vee X_m)$ is reducible with $N=\max\{\NEh(X_1),\cdots,\NEh(X_m)\}$, then 
	$\Ah{N}(X_1\vee \cdots\vee X_m)=\E(X_1\vee \cdots\vee X_m)$.	
\end{theorem}

\section{More on reducibility of $\A{n}(X\times Y)$}\label{sec:redu}
                    
In this section we further investigate conditions for the reducibility of the monoids $\A{n}(X\times Y)$.

Given a self-map $f$ of $X$, the induced endomorphism $\pi_k(f)$ is said to be \emph{nilpotent-central} if $\pi_k(f)$ is nilpotent and \emph{central} in the sense that it commutes with $\pi_k(g)$ for any $g\in [X,X]$. Let $R$ be a ring with the identity $1$. It is well-known that nilpotent elements in a unitary ring are \emph{quasi-regular}; that is, if $x\in R$ satisfies $x^n=0$, then $1-x$ (or $1+x$) is a unit.
We say an endomorphism of $R$ is \emph{radical} if it belongs to the Jacobson radical $J(R)$ of $R$, which consists of elements $x$ of $R$ such that $1+rxs$ is a unit for any $r,s\in R$.

\begin{lemma}\label{lem:NC}
	Let $a=u+t$ be an equation in a ring $(R,+,\cdot,1)$, where $u$ is a unit and $t$ is a nilpotent. If $at=ta$, then $a$ is a unit. 
	
	\begin{proof}
		It is direct to check that $at=ta$ is equivalent to $ut=tu$, by the equality $a=u+t$. Then $u^{-1}t=u^{-1}(tu)u^{-1}=u^{-1}(ut)u^{-1}=tu^{-1}$ is nilpotent, and hence $a=u(1+u^{-1}t)$ is a unit.
	\end{proof}
\end{lemma}

The following is a basic theorem of this section, other results are derived from it or its proof given here.

\begin{theorem}\label{thm:redu-prod}
Let $n\geq \NE(X)$. 
Suppose that every self-map of $Y$ that factors through $X$ induces nilpotent-central or radical endomorphisms of the first $n$ homotopy groups of $Y$.
Then $\A{n}(X\times Y)$ is reducible if and only if $f\in\A{n}(X\times Y)$ implies that $f_{XX}\in\E(X)$.
 
\begin{proof}
	Suppose $f=(f_X,f_Y)\in\A{n}(X\times Y)$ and $f_{XX}\in\E(X)$. It suffices to show that $f_{YY}\in\A{n}(Y)$. Let $g\in\E(X)$ be the homotopy inverse of $f_{XX}$.  For each $k\leq n$, let 
	\[\phi_k=[\phi_{ij}]=\begin{bmatrix}
		\phi_{11} & \phi_{12} \\
		\phi_{21} & \phi_{22}
	\end{bmatrix}\in\Aut(\pi_k(X)\oplus\pi_k(Y))\]
	be the inverse of $\pi_k(f)$. The matrix multiplication $\pi_k(f)\cdot\phi_k =\id_{\pi_k(X)\oplus\pi_k(Y)}$ implies that 
    \begin{align*}
		\pi_k(f_{XX})\phi_{12}+\pi_k(f_{XY})\phi_{22}&=0,\\
		\pi_k(f_{YX})\phi_{12}+\pi_k(f_{YY})\phi_{22}&=\id_{\pi_k(Y)}.
	\end{align*}
	Then $\phi_{12}=-\pi_k(g)\pi_k(f_{XY})\phi_{22}$ and hence 
	\[[\pi_k(f_{YY})-\pi_k(f_{YX}\circ g\circ f_{XY})]\phi_{22}=\id_{\pi_k(Y)}.\]
    Similarly, from the matrix multiplication $\phi_k\cdot \pi_k(f)=\id_{\pi_k(X)\oplus\pi_k(Y)}$ we deduce that 
	\[\phi_{22}[\pi_k(f_{YY})-\pi_k(f_{YX}\circ g\circ f_{XY})]=\id_{\pi_k(Y)}.\]
    It follows that $\phi_{22}\in\Aut(\pi_k(Y))$ and
    \begin{equation}\label{eq:NC}
    \pi_k(f_{YY})=\phi_{22}^{-1}+\pi_k(f_{YX}\circ g\circ f_{XY})=\phi_{22}^{-1}\big(\id_{\pi_k(Y)}+\phi_{22}\pi_k(f_{YX}\circ g\circ f_{XY})\big).\tag{$\ast$}
    \end{equation}

	If $\pi_k(f_{YX}\circ g\circ f_{XY})$ is radical, $\id_{\pi_k(Y)}+\phi_{22}\pi_k(f_{YX}\circ g\circ f_{XY})$ is a unit of $\End(\pi_k(Y))$, and hence $\pi_k(f_{YY})$ is invertible.
	If $\pi_k(f_{YX}\circ g\circ f_{XY})$ is nilpotent-central, the endomorphisms induced by $f_{YY}$ and $f_{YX}\circ g\circ f_{XY}$ commute, thence by Lemma \ref{lem:NC} we get that $\pi_k(f_{YY})$ is invertible.
\end{proof}
\end{theorem}

If the image $\mathrm{im}(\beta_n^Y)$ of the homotopy representation 
\[\beta_n^Y\colon [Y,Y]\to \prod_{k=1}^n\pi_k(Y),\quad f\mapsto (\pi_1(f),\cdots,\pi_n(f))\] 
is a commutative subring, Theorem \ref{thm:redu-prod} can be modified as follows. 

\begin{corollary}\label{cor:image-commute}
Let $n\geq \NE(X)$. 
Suppose that $\mathrm{im}(\beta_n^Y)$ is commutative and that every self-map of $Y$ that factors through $X$ induces nilpotent endomophisms of the first $n$ homotopy groups of $Y$.
Then $\A{n}(X\times Y)$ is reducible if and only if $f\in\A{n}(X\times Y)$ implies that $f_{XX}\in\E(X)$.

\end{corollary}

Projective spaces and Lens spaces of different dimensions satisfy such the second property described in Corollary \ref{cor:image-commute}.

\begin{lemma}\label{lem:projSP}
	Let $\mathbb{F}=\mathbb{R,C,H}$ and denote $d=\dim_\R(\mathbb{F})=1,2,4$, respectively. Set $\ZF=\Z/2$ if $\mathbb{F}=\R$; $\ZF=\Z$ if $\mathbb{F}=\mathbb{C,H}$.  There hold some basic facts. 	
	\begin{enumerate}[(1)]
		\item\label{P1}  Partial homotopy groups of $\FP{m}~(m\geq 2)$ are give by:
	 \[\pi_k(\FP{n})\cong \left\{\begin{array}{ll}
		\ZF&k=d;\\
		 \Z&k=d(n+1)-1;\\
		 0& k<d\text{ or } d<k<d(n+1)-1.
	 \end{array}\right.\]
	 \item\label{P2} If $2\leq m<n\leq \infty$, then every map $\varphi\colon \FP{n}\to\FP{m}$ induces a trivial homomorphism $\varphi^\ast\colon H^\ast(\FP{m};\ZF)\to H^\ast(\FP{n};\ZF)$.
	 \item\label{P3} For any $m,n$, there hold a chain of natural isomorphisms 
	    \[ \pi_d(\FP{n})\cong H_d(\FP{n})\cong H_d(\FP{n};\ZF) \cong H^d(\FP{n};\ZF).\] 
	\end{enumerate}
	\begin{proof}
		(\ref{P1}) follows by the associated long exact sequence of homotopy groups associated to the Hopf fibrations
		\[S^{d-1}\to S^{d(n+1)-1}\to \FP{n}.\]

		(\ref{P2}) is a direct consequence of naturality of ring isomorphismsm \[H^\ast(\FP{m};\ZF)\cong \ZF[x]/(x^{dm+1}).\]

		For the chain in (\ref{P3}), the first natural isomorphism is induced by the Hurewicz map: if $\mathbb{F}=\R$, $\pi_1(\R P^m)\cong H_1(\R P^m)$ for $\pi_1(\R P^m)\cong \Z/2$; if $\mathbb{F}=\mathbb{C,H}$, the first natural isomorphism is due to the Hurewicz theorem. The second and the third natural isomorphisms hold by the universal coefficient theorems for homology and cohomology, respectively. 
	\end{proof}
\end{lemma}

The self-closeness numbers of projective spaces $\FP{n}$ over fields $\mathbb{F=R,C,H}$ and lens spaces $L^{2n+1}(p)=S^{2n+1}/\zp{}$ with $p$ a prime are known, see \citep[Theorem 6]{OY17} and \citep[Theorem 13,14]{OY20}.

\begin{example}\label{ex:ProjSp}
Let $2\leq m<n$ be positive integers.
\begin{enumerate}
	\item $\NE(\FP{m}\times \FP{n})=\NE(\FP{n})=n,2,4$ for $\mathbb{F}=\R,\mathbb{C,H}$, respectively.
	\item $\NE(L^{2m+1}(p)\times L^{2n+1}(q))=2n+1$, where $p,q$ are primes.
\end{enumerate}

\begin{proof}
(1) By Theorem \ref{thm:scn-prod} it suffices to show that $\A{n}(\FP{m}\times \FP{n})$ are reducible. Suppose $f=(f_1,,f_2)\in\A{n}(\FP{m}\times \FP{n})$. By Lemma \ref{lem:projSP} (\ref{P3}), $H^d(f;\ZF)\in \Aut\big(H^d(\FP{m}\times \FP{n};\ZF)\big)$. By Lemma \ref{lem:projSP} (\ref{P2}), the homomorphisms $f_{12}\colon \FP{n}\to \FP{m}$ induces a trivial endomorphism on $H^d(-;\ZF)$, hence $H^d(f;\ZF)$ can be represented by a lower-triangular matrix, and we then get \[H^d(f_{11};\ZF)\in \Aut\big(H^d(\FP{m};\ZF)\big),\quad H^d(f_{22};\ZF)\in \Aut\big(H^d(\FP{n};\ZF)\big),\]
which are equivalent to $\pi_d(f_{11}),\pi_d(f_{22})$ are automorphisms.

If $\mathbb{F}=\mathbb{C,H}$, we are done. If $\mathbb{F}=\R$, firstly $f\in\A{n}(\R P^{m}\times \R P^{n})\subseteq \A{m}(\R P^{m}\times \R P^{n})$ implies that $f_{11}\in \A{m}(\R P^{m})=\E(\R P^{m})$.
By Lemma \ref{lem:projSP} (\ref{P1}), we need to show that every composition $\R P^{n}\to \R P^{m}\to \R P^{n}$ induces a zero endomorphism on $\pi_n(-)$. If $\pi_n(\R P^{m})\cong \pi_n(S^m)$ is finite, this is clear. So by Serre's finiteness theorem on homotopy groups of spheres, it suffices to consider the case $m$ is even and $n=2m-1$. Note that in this case, the Hurewicz homomorphism $\pi_n(\R P^n)\to H_n(\R P^n)$ is a monomorphism. Consider the following commutative diagram induced by the Hurewicz homomorphism:
\[\begin{tikzcd}
	\pi_n(\R P^n)\ar[d,tail]\ar[r]&\pi_n(\R P^m)\ar[r]\ar[d,tail]&\pi_n(\R P^n)\ar[d,tail]\\
	H_n(\R P^n)\ar[r]&H_n(\R P^m)=0\ar[r]&H_n(\R P^n)
\end{tikzcd}\] 
It follows that the endomorphism of $\pi_n(\R P^n)$ induced by self-maps of $\R P^n$ that factors through $\R P^m$ is trivial.
Thus applying Theorem \ref{thm:redu-prod} we get that $\A{n}(\R P^{m}\times \R P^{n})$ is reducible. 

(2) By Theorem \ref{thm:redu-prod} and the arguments in (1) with $\mathbb{F=R}$, it suffices to show that every self-map of $L^{2n+1}(p)$ that factors through $L^{2m+1}(q)$ induces a trivial endomorphism on $\pi_{2n+1}(-)$. This is clear since $\pi_{2n+1}(L^{2m+1}(q))\cong \pi_{2n+1}(S^{2m+1})$ is finite.
	\end{proof}
\end{example}

 The following truncated versions of the concepts \emph{homotopically/homologically distance} are essentially due to \Pave \cite{Pavesic07}.

\begin{definition}\label{def:n-dist}
	Let $X,Y$ be two complexes and let $1\leq n\leq \infty$ be fixed. We say
	 $X$ and $Y$ are \emph{homotopically} (resp. \emph{homologically}) \emph{$n$-distant} if every self-map of $Y$ that factors through $X$ induces a nilpotent endomorphism of $\pi_k(Y)$ (resp. $H_k(Y)$) for each $k\leq n$.

	 If $n=\infty$, we just say $X$ and $Y$ are homotopically/homologically distant. 
	\end{definition}

Both the above relations are symmetric for each $1\leq n\leq \infty$ and $n$-distance implies $m$-distance for any $n\geq m$. Moreover, if $\pi_1(X)$ is solvable or $\pi_1(f)$ is nilpotent, then for each $n$, $X$ and $Y$ are homotopically $n$-distant if and only if $X$ and $Y$ are homologically $n$-distant (cf. \citep[Theorem 3.5, Lemma 3.4]{Pavesic07}).

Give a unital ring $R$, denote by $N(R)$ the set of nilpotent elements of $R$. There are some notions in ring theory. We say that $R$ is \emph{reduced} if $N(R)=\{0\}$,  is \emph{central reduced} if $N(R)\subseteq Z(R)$, where $Z(R)$ denotes the center of $R$, and is \emph{$J$-reduced} if $N(R)\subseteq J(R)$. Reduced ring, central reduced rings are obviously $J$-reduced, and rings with these properties were well studied, such as \cite{UHKH13,CGHH16,LP18}. Every commutative ring and local ring are obviously $J$-reduced, and finite (sub)direct product of $J$-reduced rings are $J$-reduced, by \citep[Corollary 2.10]{CGHH16}.  
If the endomorphism ring $\prod_{k=1}^n\End(\pi_k(Y))$ is $J$-reduced, then the assumptions in Theorem \ref{thm:redu-prod} can be modified as follows. 
  
\begin{proposition}\label{prop:J-redu}
	Let $n\geq \NE(X)$ and let $Y$ be a space such that the endomorphism ring $\End(\pi_k(Y))$ is $J$-reduced for each $k\leq n$. If $X$ and $Y$ are homotopically $n$-distant, then $\A{n}(X\times Y)$ is reducible if and only if $f\in\A{n}(X\times Y)$ implies that $f_{XX}\in\E(X)$.
\end{proposition}

Let $G=\oplus_p G_p$ be a finite abelian group, where $G_p$ are the $p$-primary components. Then $\End(G)\cong \oplus_p\End(G_p)$. Note that $\End(G)$ is $J$-reduced if and only if so is each $\End(G_p)$. For a $p$-group $H$, denote by $p^sH[p]=\{p^s x:x\in H, p^{s+1}x=0\}$ and let $f_s(H)~(k\geq 0)$ be the $s$-th \emph{Ulm-Kaplansky invariant} of $H$ defined by 
\begin{equation*}
f_s(H)=\dim_{\zp{}}\big(p^sH[p])/p^{s+1}H[p]\big).	
\end{equation*}

\begin{lemma}\label{lem:NJ}
	Let $H\cong \zp{r_1}\oplus\cdots\oplus\zp{r_m}$, $1\leq r_1\leq \cdots\leq r_m<\infty$. Let $f_s(H)$ be defined as above, $s\geq 0$.
	\begin{enumerate}
		\item If $\End(H)/J(\End(H))$ is reduced, then $\End(H)$ is $J$-reduced.
		\item $\End(H)/J(\End(H))$ is reduced if and only if $f_s(H)=0$ or $1$ for each $s\geq 0$.
		\item For each $s\geq 0$, $f_s(H)=\#\{r_i:r_i>s\}-\#\{r_i:r_i>s+1\}$, where $\# S$ denotes the cardinality of the set $S$.
		\item $f_s(H)=0$ or $1$ for each $s\geq 0$ if and only if $H$ is a subgroup of $\bigoplus_{i\geq 1}\zp{i}$.
	\end{enumerate}
	\begin{proof}
		(1) refers to \citep[Lemma 2.2]{CGHH16}.

		(2) By \citep[Corollary 20.14]{KMT03}, there holds a ring isomorphism 
		\[\End(H)/J(\End(H))\cong \prod_{k\geq 0}M_{f_s(H)}(\zp{}),\]
		where $M_{t}(\zp{})$ denotes the ring of  $t\times t$ matrices over the field $\zp{}$.
        The equivalence in the lemma then follows from the fact that $M_{f_s(H)}(\zp{})$ is reduced if and only if $f_s(H)=0,1$. 

		(3) For any integers $s$, there is an isomorphism of groups 
		\[p^sH\cong\bigoplus_{i=1}^m p^s\zp{r_i}\cong\bigoplus_{i=1}^m\zp{r_i-s},\] 
		where $\zp{r_i-s}=0$ if $r_i\leq s$ for each $i=1,\cdots,m$. It follows that the dimension of $p^sH[p]$ equals to the number of $r_i$ that are greater than $s$. The equality then follows by the definition of $f_s(H)$.

		(4) If $H\leq\bigoplus_{i=1}^{\infty}\zp{i}$, then the powers $r_i$ satisfy $r_1<\cdots<r_m$. By (3) we compute that 
		\[f_{s}(H)=\left\{\begin{array}{ll}
		1&s=r_i-1,i=1,\cdots,m;\\
		0&\text{otherwise}.
		\end{array} \right.\]

Conversely, assume that $H$ contains two cyclic direct summand of the same orders, say $r_1=r_2$. Then $\#\{r_i:r_i>r_1\}\leq m-2$ and hence\[f_{r_1-1}(H)=m-\#\{r_i:r_i>r_1\}\geq 2.\]
	\end{proof}
\end{lemma}

\begin{proposition}\label{prop:NJ}
	Let $n\geq \NE(X)$ and let the $p$-primary component $\pi_k(Y;p)$ of $\pi_k(Y)$ be isomorphic to a subgroup of $\bigoplus_{i=1}^{\infty}\zp{i}$ for each prime $p$ and each $k\leq n$.
	If $X$ and $Y$ is homotopically $n$-distant, then $\A{n}(X\times Y)$ is reducible if and only if $f\in\A{n}(X\times Y)$ implies that $f_{XX}\in\E(X)$.
\begin{proof}
	By Proposition \ref{prop:J-redu} and Lemma \ref{lem:NJ}.
\end{proof}
\end{proposition}

A special case of Proposition \ref{prop:NJ} is that each $p$-primary component $\pi_k(Y;p)$ is a cyclic group. In this case, $\End(\pi_k(Y))$ is actually commutative, by \citep[Theorem 1]{SS51}. Thus we have

\begin{corollary}\label{cor:grp-commute}
	Suppose that $\pi_k(Y)\cong \bigoplus_{p\in S_k}\zp{r_p}$ for each $k\leq n$, where $S_k$ is a set of (different) primes, and that $X$ and $Y$ are homotopically $n$-distant, $n\geq \NE(X)$. 
	Then $\A{n}(X\times Y)$ is reducible if and only if $f\in\A{n}(X\times Y)$ implies that $f_{XX}\in\E(X)$. 
\end{corollary}

\section{Applications}\label{sec:Appl}

This section applies discussions in Section \ref{sec:redu} to determine self-closeness numbers of products of some special spaces. 

\subsection{$X=M_n(G)$ or $X=K_n(G)$}\label{sec:Moore-EM}
If every homomorphism $\pi_n(X)\to\pi_n(Y)$ is induced by some map $X\to Y$, then we can reduce assumptions in Theorem \ref{thm:redu-prod}. In this subsection we mainly consider the cases where $X$ is a Moore space or an Eilenberg-MacLane space.

Given an abelian group $G$, denote by $M_n(G)$ the Moore space with a unique nontrivial reduced integral homology group $G$ in dimension $n$. It is clear that $\NEh(M_n(G))=n$ for any $n\geq 1$.
The following is an immediate result of Theorem \ref{thm:scn-sum}.
\begin{example}\label{ex:Moore}
	Let $n_1,\cdots,n_m\geq 2$ be positive integers and let $G_1,\cdots,G_m$ be abelian groups. Then 
	\[\NEh(M_{n_1}(G_1)\vee \cdots \vee M_{n_m}(G_m))=\max\{n_1,\cdots,n_m\}.\]
\end{example}

The universal coefficient theorem for homotopy (cf.\citep[Theorem 3.8]{Katuta60} or \citep[Proposition 1.3.4]{Bauesbook}) tells that for any space $Y$, there is an epimorphism of groups for each $n\geq 2$:   
\begin{equation}
\begin{tikzcd}
	{[M_n(G),Y]}\ar[r,two heads,"\pi_n(-)"]&\Hom(G,\pi_n(Y)).
 \end{tikzcd}\tag{$\natural$}\label{UCT}
\end{equation}
Taking $Y=M_n(G)$, we get $\NE(M_n(G))=n$ for any $n\geq 2$.

\begin{proposition}\label{prop:Moore}
	Let $G$ be a finitely generated abelian group. If $M_n(G)$ and $Y$ are homologically (or homotopically) $n$-distant, $n\geq 2$, then $\A{n}(M_n(G)\times Y)$ is reducible. If, in addition, $\NE(Y)\leq n$, then $\NE(M_n(G)\times Y)=n$, by Theorem \ref{thm:scn-prod}.

	\begin{proof}
	Suppose that $f=(f_M,f_Y)\in\A{n}(M_n(G)\times Y)$ and let $\phi=[\phi_{ij}]_{2\times 2}$ be the inverse of $\pi_n(f)$. The matrix multiplication $\pi_n(f)\cdot \phi=\id_{G\oplus\pi_n(Y)}$ then implies 
	\begin{align}\label{eq:Moore}
	\pi_n(f_{MM})\phi_{11}&+\pi_n(f_{MY})\phi_{21}=\id_G,\notag\\
	\pi_n(f_{YM})\phi_{11}&+\pi_n(f_{YY})\phi_{21}=0,\tag{$\ast\ast$}\\
	\pi_n(f_{YM})\phi_{12}&+\pi_n(f_{YY})\phi_{22}=\id_{\pi_n(Y)}.\notag
	\end{align}
    By the surjection (\ref{UCT}), $\phi_{21}=\pi_n(u)$ for some $u\in [M_n(G),Y]$, and hence by the first equation of (\ref{eq:Moore}) we have  
	\[\pi_n(f_{MM})\phi_{11}=\id_G-\pi_n(f_{MY}\circ u)\in\Aut(G).\]
	Since $G$ is finitely generated, both $\pi_n(f_{MM})$ and $\phi_{11}$ are automorphisms. 

    By the last two equations of (\ref{eq:Moore}) we get
	\[\pi_n(f_{YY})(\phi_{22}-\phi_{21}\phi_{11}^{-1}\phi_{12})=\id_{\pi_n(Y)}.\]
	Similarly, the matrix multiplication $\phi\cdot \pi_n(f)=\id_{G\oplus\pi_n(Y)}$ implies 
	\[(\phi_{22}-\phi_{21}\phi_{11}^{-1}\phi_{12})\pi_n(f_{YY})=\id_{\pi_n(Y)}.\]
	Thus $\pi_n(f_{YY})\in\Aut(\pi_n(Y))$, which completes the proof. 
	\end{proof}
\end{proposition}

\begin{corollary}\label{cor:prod-Moore}
Let $G_1,\cdots,G_m$ be finitely generated abelian groups and let $n_1,\cdots,n_m$ be mutually different integers that are greater than $1$. Then  
\[\NE\big(M_{n_1}(G_1)\times\cdots\times M_{n_m}(G_m)\big)=\max\{n_1,\cdots,n_m\}.\]
\begin{proof}
	We may assume that $n_1>\cdots>n_m$. Inductively applying Proposition \ref{prop:Moore} it suffices to show that 
	$M_{n_1}(G_1)$ and $Y=M_{n_2}(G_2)\times\cdots\times M_{n_m}(G_m)$ is homotopically $n_1$-distant.

	The case $m=2$ is clear, since $M_{n_1}(G_1)$ and $M_{n_2}(G_2)$ are homologically distant. If $m\geq 3$, given maps $f\colon M_{n_1}(G)\to Y$ and $g\colon Y\to M_{n_1}(G_1)$, we have \[\pi_{n_1}(g\circ f)=\pi_{n_1}(g\circ e_2\circ f)+\cdots+\pi_{n_1}(g\circ e_m\circ f),\] where $e_j=i_j\circ p_j$ is the idempotent of $Y$ that factors through $M_{n_j}(G_j)$, $j=2,\cdots,m$. Each endomorphism $H_{n_1}(g\circ e_j\circ f)$ on the homology group $H_{n_1}(M_{n_1}(G_1))$ is trivial, the naturality of the Hurewicz theorem then implies so are $\pi_{n_1}(g\circ e_j\circ f)$ , $j=2,\cdots,m$. Thus $\pi_{n_1}(g\circ f)=0$, and therefore $M_{n_1}(G_1)$ and $Y$ is homotopically $n_1$-distant.
\end{proof}
\end{corollary}

\begin{corollary}\label{cor:EM}
	Let $G$ be a finitely generated abelian group and let $Y$ be a space such that $\pi_j(Y)=0$ for $j>n$, $n\geq 2$. If $K_n(G)$ and $Y$ is homotopically $n$-distant, then $\NE(K_n(G)\times Y)=n$.
	\begin{proof}
	Since $\pi_j(Y)=0$ for $j>n$, $\NE(Y)\leq n$, and the canonical inclusion map $M_n(G)\to K_n(G)$ induces a bijection (cf. \citep[Proposition 2.4.13]{ArkBook}): 
	\[[K_n(G),Y]\xra{\cong}[M_n(G),Y].\]
Composing the surjection (\ref{UCT}) we get a surjection
	\[[K_n(G),Y] \to \Hom(G,\pi_n(Y)).\]
	
	The proof of the reducibility of $\A{n}(K_n(G)\times Y)$ is then totally parallel to that of Proposition \ref{prop:Moore}. 
	\end{proof} 
\end{corollary}

\subsection{Products of atomic spaces}\label{sec:atomic}
In the sense of Baker and May \cite{BM04}, a $p$-local CW-complex $X$ is \emph{atomic of Hurewicz dimension $n_0$} if $\pi_{k<n_0}(X)=0$ and $\pi_{n_0}(X)$ is a nonzero cyclic module over $\Z_{(p)}$, and every self-map of $X$ inducing an automorphism on $\pi_{n_0}(X)$ is a homotopy equivalence. The atomicity of spaces appeared in the study of exponents of homotopy groups of $p$-local spheres and mod $p^r$ Moore spaces \cite{CMN}. Note that in this case $\NE(X)=n_0$ and $\End(\pi_{n_0}(X))$ is local. An important property of atomic spaces is that they are 
``prime" in the following sense.

\begin{lemma}\label{lem:atomic}
	Let $X$ be an atomic space of Hurewicz dimension $n$. If $X$ is a retract of $Y\times Z$, then $X$ is a retract of $Y$ or $Z$.
	\begin{proof}
	Let $e_J$ be the idempotent self-map of $Y\times Z$ that factors through $J$,  $J\in \{Y,Z\}$; that is, $e_I$ is the composition given by
	\[Y\times Z\xra{p_I}I\xra{i_J} Y\times Z.\]
	 Let $f_I\colon X\to X$ be the composition given by 
	\[X\xra{\phi}Y\times Z\xra{e_I}Y\times Z\xra{\psi}X, \]
	where $\phi,\psi$ satisfies $\psi\phi=\id_{X}$. Then the formula $\pi_n(e_Y)+\pi_n(e_Z)=\id$ implies that  \[\pi_n(f_Y)+\pi_n(f_Z)=\id_{\pi_n(X)}.\]
	 Since $\End(\pi_n(X))$ is local, either $\pi_n(f_Y)$ or $\pi_n(f_Z)$ is an automorphism of $\pi_n(X)$. $\NE(X)=n$ then implies that either $f_Y$ or $f_Z$ is a homotopy equivalence, and therefore $X$ is a retract of $Y$ or $Z$.
	\end{proof}
\end{lemma}

\begin{theorem}\label{thm:atomic}
	Let $X_1,\cdots,X_m$ be atomic spaces of Hurewicz dimensions $n_1,\cdots,n_m$, respectively. If for any $i\neq j$, $n_i\neq n_j$ and $X_i$ is not a retract of $X_j$, then \[\NE(X_1\times\cdots\times X_m)=\max\{n_1,\cdots,n_m\}.\]

\begin{proof}
Arranging the Hurewicz dimensions $n_i$ such that $n_1<\cdots<n_m$, it suffices to show that $\A{n_m}(X_1\times\cdots\times X_m)$ is reducible. 

If $m=2$, $n_1<n_2$ implies that $\A{n_2-1}(X_1\times X_2)$ is reducible and $f=(f_1,f_2)\in\A{n_2}(X_1\times X_2)\subseteq \A{n_1}(X_1\times X_2)$ is equivalent to $f_{11}\in\A{n_1}(X_1)=\E(X_1)$. Suppose that $\pi_{n_2}(f)=(f_{1\sharp},f_{2\sharp})\in\Aut(\pi_{n_2}(X_1\times X_2))$. Let $g\in\E(X_1)$ be the inverse of $f_{11}$.  Using the arguments in the proof of Theorem \ref{thm:redu-prod}, we have  
\begin{align*}
	f_{21\sharp}\phi_{12}+f_{22\sharp}\phi_{22}=\id_{\pi_{n_2}(X_2)};\\
\phi_{12}=-(g f_{12})_\sharp\phi_{22},\quad
\phi_{22}\in\Aut(\pi_{n_2}(X_2)).	
\end{align*}	
Since $\pi_{n_2}(X_2)$ is local, the above first equation implies that $f_{21\sharp}\phi_{12}$ or $f_{22\sharp}\phi_{22}$ is an automorphism.  If $f_{21\sharp}\phi_{12}=-f_{21\sharp}(g f_{12})_\sharp\phi_{22}$ is an automorphism, then $(f_{21}gf_{12})_\sharp$ is an automorphism of $\pi_{n_2}(X_2)$. Since $X_2$ is atomic, $f_{21}gf_{12}\in\E(X_2)$, which in turn implies that $X_2$ is a retract of $X_1$, contradiction. It follows that $f_{22\sharp}\phi_{22}$, or equivalently $f_{22\sharp}$ is an automorphism.   

If $m\geq 3$, inductively applying Lemma \ref{lem:atomic} we see that $X_m$ is not a retract of $X=X_1\times\cdots\times X_{m-1}$. By induction and totally similar arguments, $f\in\A{n_m}(X\times X_m)$ implies that $f_{XX}\in\E(X)$ and $f_{X_mX_m}\in\E(X_m)$. Thus $\A{n_{m}}(X_1\times \cdots\times X_m)$ is reducible, the proof completes. 
\end{proof}		
\end{theorem}

\bibliographystyle{amsplain}
\bibliography{refs}
\end{document}